\documentclass{amsart}
\usepackage{graphicx}
\usepackage{amsmath}
\usepackage{amssymb}
\usepackage{oldgerm}
\usepackage{mathdots}
\usepackage{stmaryrd}
\usepackage{bm}
\usepackage[all]{xy}
\usepackage{color}

\newcommand{\Hom}{\operatorname{Hom}\nolimits}
\renewcommand{\Im}{\operatorname{Im}\nolimits}

\newcommand{\Coker}{\operatorname{Coker}\nolimits}

\newcommand{\n}{\mathfrak{n}}

\newcommand{\pd}{\operatorname{pd}\nolimits}
\newcommand{\proj}{\operatorname{proj}\nolimits}

\newcommand{\V}{\operatorname{V}\nolimits}

\newcommand{\MF}{\operatorname{\bf{MF}}\nolimits}
\newcommand{\HMF}{\operatorname{\bf{HMF}}\nolimits}

\newcommand{\MCM}{\operatorname{\bf{MCM}}\nolimits}
\newcommand{\stMCM}{\operatorname{\underline{\bf{MCM}}}\nolimits}
\newcommand{\derived}{\operatorname{\bf{D^{b}}}\nolimits}
\newcommand{\sing}{\operatorname{\bf{D_{sg}}}\nolimits}
\newcommand{\Kac}{\operatorname{\bf{K_{ac}}}\nolimits}
\newcommand{\Ktac}{\operatorname{\bf{K_{tac}}}\nolimits}
\newcommand{\thick}{\operatorname{\bf{thick}}\nolimits}
\newcommand{\comments}[1]{}

\newtheorem{theorem}{Theorem}[section]

\newtheorem{corollary}[theorem]{Corollary}

\newtheorem{lemma}[theorem]{Lemma}
\newtheorem{proposition}[theorem]{Proposition}

\theoremstyle{definition}

\theoremstyle{definition}

\newtheorem{example}[theorem]{Example}
\theoremstyle{definition}

\theoremstyle{remark}

\theoremstyle{remark}

\theoremstyle{definition}

\begin{document}

\title[Complete intersections and matrix factorizations]{Complete intersections and equivalences with categories of matrix factorizations}

\author{Petter Andreas Bergh and David A. Jorgensen}

\address{Petter Andreas Bergh \\ Institutt for matematiske fag \\
NTNU \\ N-7491 Trondheim \\ Norway} \email{bergh@math.ntnu.no}
\address{David A.\ Jorgensen \\ Department of Mathematics \\ University
of Texas at Arlington \\ Arlington \\ TX 76019 \\ USA}
\email{djorgens@uta.edu}



\begin{abstract}
We prove that one can realize certain triangulated subcategories of the singularity category of a complete intersection as homotopy categories of matrix factorizations. Moreover, we prove that for any commutative ring and non-zerodivisor, the homotopy category of matrix factorizations embeds into the homotopy category of totally acyclic complexes of finitely generated projective modules over the factor ring. 
\end{abstract}

\maketitle

\section{Introduction}\label{Sec:intro}

Matrix factorizations of elements in commutative rings were introduced by Eisenbud in \cite{Eisenbud}, in order to study free resolutions over the corresponding factor rings. In particular, he showed that minimal free resolutions over hypersurface rings eventually correspond to matrix factorizations, and are therefore eventually two-periodic. More precisely, let $Q$ be a regular local ring, $x$ a nonzero element, and denote the factor ring $Q/(x)$ by $R$. Eisenbud showed that if we take any finitely generated maximal Cohen-Macaulay module over $R$, without free summands, then its minimal free resolution is obtained from a matrix factorization of $x$ over $Q$.

The homotopy category of all matrix factorizations of $x$ over $Q$ forms a triangulated category in a natural way. The distinguished triangles are those that are isomorphic to the standard triangles constructed using mapping cones, as in the homotopy category of complexes over an additive category. Buchweitz remarked in \cite{Buchweitz} that in the above situation, with $Q$ regular, the homotopy category of matrix factorizations of $x$ is equivalent to the singularity category of $R$; this was proved explicitly by Orlov in \cite[Theorem 3.9]{Orlov1}. 

In this paper, we prove an analogue of the Buchweitz-Orlov result for more general complete intersections. Let $R$ be a complete intersection of codimension $c$, with $c \ge 2$. We prove that the singularity category of $R$ contains several triangulated subcategories that are equivalent to homotopy categories of matrix factorizations over complete intersections of codimension $c-1$. 

Along the way we also prove a more general embedding result analogous to \cite[Theorem 1]{Orlov2} and \cite[Example B.5]{BurkeWalker2}. It shows that for \emph{any} commutative ring and non-zerodivisor, the homotopy category of matrix factorizations embeds into the homotopy category of totally acyclic complexes of finitely generated projective modules over the factor ring.

\section{Preliminaries}\label{Sec:pre}

\subsection*{Matrix factorizations}

Let $S$ be a commutative ring and $x$ an element of $S$. A \emph{matrix factorization} $(F,G,\phi,\psi)$ of $x$ is a diagram
$$\xymatrix@C=30pt{
F \ar[r]^{\phi} & G \ar[r]^{\psi} & F}$$
in which $F$ and $G$ are finitely generated free $S$-modules, and $\phi$ and $\psi$ are $S$-homomorphisms satisfying
\begin{eqnarray*}
\psi \circ \phi & = & x \cdot 1_F \\
\phi \circ \psi & = & x \cdot 1_G.
\end{eqnarray*}
A morphism $\theta \colon (F_1,G_1,\phi_1,\psi_1) \to (F_2,G_2,\phi_2,\psi_2)$ between two matrix factorizations (of $x$) is a pair of homomorphisms $f \colon F_1 \to F_2$ and $g \colon G_1 \to G_2$ such that the diagram
$$\xymatrix@C=30pt@R=20pt{
F_1 \ar[r]^{\phi_1} \ar[d]^{f} & G_1 \ar[r]^{\psi_1}  \ar[d]^{g} & F_1  \ar[d]^{f} \\
F_2 \ar[r]^{\phi_2} & G_2 \ar[r]^{\psi_2} & F_2}$$
commutes. The category $\MF(S,x)$ of matrix factorizations and maps is additive, with the obvious notion of a zero object and direct sums.

The \emph{suspension} $\Sigma (F,G,\phi,\psi)$ of $(F,G,\phi,\psi)$ is the matrix factorization
$$\xymatrix@C=30pt{
G \ar[r]^{- \psi} & F \ar[r]^{- \phi} & G}$$
of $x$. The \emph{mapping cone} $C_{\theta}$ of the map $\theta$ above is the diagram
$$\xymatrix@C=40pt{
G_1 \oplus F_2 \ar[r]^{\left [ \begin{smallmatrix}- \psi_1 & 0 \\ g & \phi_2 \end{smallmatrix} \right ]} & F_1 \oplus G_2 \ar[r]^{\left [ \begin{smallmatrix}- \phi_1 & 0 \\ f & \psi_2 \end{smallmatrix} \right ]} & G_1 \oplus F_2 }$$ 
which is easily seen to be a matrix factorization of $x$. Note that there are natural maps 
$$\xymatrix@C=48pt{
(F_2,G_2,\phi_2,\psi_2) \ar[d]^{i_{\theta}} & F_2 \ar[r]^{\phi_2} \ar[d]^{\left [ \begin{smallmatrix}0 \\ 1_{F_2} \end{smallmatrix} \right ]} & G_2 \ar[r]^{\psi_2} \ar[d]^{\left [ \begin{smallmatrix}0 \\ 1_{G_2} \end{smallmatrix} \right ]} & F_2 \ar[d]^{\left [ \begin{smallmatrix}0 \\ 1_{F_2} \end{smallmatrix} \right ]} \\
C_{\theta} & G_1 \oplus F_2 \ar[r]^{\left [ \begin{smallmatrix}- \psi_1 & 0 \\ g & \phi_2 \end{smallmatrix} \right ]} & F_1 \oplus G_2 \ar[r]^{\left [ \begin{smallmatrix}- \phi_1 & 0 \\ f & \psi_2 \end{smallmatrix} \right ]} & G_1 \oplus F_2 }$$ 
and 
$$\xymatrix@C=45pt{
C_{\theta} \ar[d]^{\pi_{\theta}} & G_1 \oplus F_2 \ar[d]^{\left [ \begin{smallmatrix}1_{G_1} & 0 \end{smallmatrix} \right ]} \ar[r]^{\left [ \begin{smallmatrix}- \psi_1 & 0 \\ g & \phi_2 \end{smallmatrix} \right ]} & F_1 \oplus G_2 \ar[d]^{\left [ \begin{smallmatrix}1_{F_1} & 0 \end{smallmatrix} \right ]} \ar[r]^{\left [ \begin{smallmatrix}- \phi_1 & 0 \\ f & \psi_2 \end{smallmatrix} \right ]} & G_1 \oplus F_2 \ar[d]^{\left [ \begin{smallmatrix}1_{G_1} & 0 \end{smallmatrix} \right ]} \\
\Sigma (F_1,G_1,\phi_1,\psi_1) & G_1 \ar[r]^{- \psi_1} & F_1 \ar[r]^{- \phi_1} & G_1 }$$
of matrix factorizations in $\MF(S,x)$.

Two maps $\theta, \theta' \colon (F_1,G_1,\phi_1,\psi_1) \to (F_2,G_2,\phi_2,\psi_2)$ in $\MF(S,x)$, with the same source and target, are \emph{homotopic} if there are diagonal maps in the diagram
$$\xymatrix@C=50pt{
F_1 \ar[r]^{\phi_1} \ar[d]^{f}_{f'} & G_1 \ar[r]^{\psi_1}  \ar[d]^{g}_{g'} \ar[dl]_{s} & F_1  \ar[d]^{f}_{f'} \ar[dl]_{t} \\
F_2 \ar[r]^{\phi_2} & G_2 \ar[r]^{\psi_2} & F_2}$$
satisfying
\begin{eqnarray*}
f - f' & = & s \circ \phi_1 + \psi_2 \circ t \\
g - g' & = & t \circ \psi_1 + \phi_2 \circ s.
\end{eqnarray*} 
This is an equivalence relation on the abelian groups of morphisms in $\MF(S,x)$, and the equivalence class of the map $\theta$ is denoted by $[ \theta ]$. It is straightforward to show that homotopies are compatible with addition and composition of maps in $\MF(S,x)$. The \emph{homotopy category} $\HMF(S,x)$ has the same objects as $\MF(S,x)$, but the morphism sets are homotopy equivalence classes of morphisms in $\MF(S,x)$. By the above, the morphism sets in $\HMF(S,x)$ are abelian groups, hence the homotopy category is also additive with the same zero object (which is now unique only up to homotopy) and the usual direct sums.

The homotopy category $\HMF(S,x)$ admits a natural structure of a triangulated category. The suspension defined above induces an additive automorphism $\Sigma \colon \HMF(S,x) \to \HMF(S,x)$, with $\Sigma^2$ the identity automorphism. Let $\Delta$ be the collection of all triangles in $\HMF(S,x)$ isomorphic to \emph{standard triangles}, that is, triangles of the form
$$\xymatrix@C=30pt{
(F_1,G_1,\phi,\psi) \ar[r]^<<<<<<{[ \theta ]} & (F_2,G_2,\phi_2,\psi_2) \ar[r]^<<<<<{[ i_{\theta} ]} & C_{\theta}  \ar[r]^<<<<<{[ \pi_{\theta} ]} & \Sigma (F_1,G_1,\phi,\psi) }$$
Then the triple $\left ( \HMF(S,x), \Sigma, \Delta \right )$ is a triangulated category; the classical proof (cf.\ \cite[Theorem 6.7]{HolmJorgensen}) showing that the homotopy category of complexes over an additive category is triangulated carries over.

\subsection*{Complete intersections}

Let $(Q, \n, k)$ be a regular local ring and $\bm{t} = t_1, \dots, t_c$ a regular sequence contained in $\n^2$. Define $R = Q/ ( \bm{t} )$; this is a \emph{complete intersection} of codimension $c$. 

Denote the $c$-dimensional $k$-vector space $(\bm{t}) / \n (\bm{t} )$ by $V$. Every basis of $V$ lifts to a regular sequence in $Q$, in particular any sequence $\overline{x}_1, \dots, \overline{x}_t$ of linearly independent elements lifts to a regular sequence $x_1, \dots, x_t$ which can be completed to a regular sequence $x_1, \dots, x_c$ generating the ideal $(\bm{t})$. Now take a single element $\overline{x}$ in $V$, lift it to an element $x \in Q$, and consider the hypersurface $Q/(x)$. By the above, the ring $R$ is a factor of this hypersurface by a regular sequence. Following \cite{BerghJorgensen}, we define the \emph{support variety} of an $R$-module $M$ as
$$\V_R(M) \stackrel{\text{def}}{=} \{ \overline{x} \in V \mid \pd_{Q/(x)} = \infty \},$$
that is, the set of all vectors in $V$ for which $M$ has infinite projective dimension over the corresponding hypersurface. By \cite[Remark following Definition 2.1]{BerghJorgensen}, this definition is well defined; if $\overline{x} = \overline{y}$ in $V$, then $\pd_{Q/(x)} = \infty$ if and only if $\pd_{Q/(y)} = \infty$. Note that $\V_R(M)$ is \emph{not} a subspace of $V$, but it is a cone, i.e.\ if $\overline{x} \in V$, then $\alpha \overline{x} \in V$ for all $\alpha \in k$. Also, the varieties just defined are isomorphic to the cohomological support varieties defined by Avramov and Buchweitz in \cite[Theorem 2.5]{AvramovBuchweitz}. 

Our main result uses support varieties to establish equivalences between certain homotopy categories of matrix factorizations and triangulated subcategories of the singularity category $\sing (R)$ of $R$. The latter category is the Verdier quotient $\derived(R)/ \thick(R)$, where $\derived(R)$ is the bounded derived category of $R$-modules, while $\thick(R)$ is the thick subcategory of $\derived(R)$ generated by $R$. In other words, $\thick(R)$ consists of the perfect complexes, i.e.\ the finite complexes whose modules are finitely generated free $R$-modules. Every complete intersection is a Gorenstein ring, and therefore the singularity category of $R$ can be interpreted both in terms of maximal Cohen-Macaulay modules and in terms of acyclic complexes of free modules. Namely, consider the category $\MCM(R)$ of maximal Cohen-Macaulay $R$-modules. This is a Frobenius category, that is, an exact category (in the sense of Quillen) with enough projectives, and the projective and injective objects coincide. Following \cite[Chapter I.2]{Happel}, we form the stable category $\stMCM (R)$, which admits a natural structure of a triangulated category with the cosyzygy functor $\Omega_R^{-1} \colon \stMCM (R) \to \stMCM (R)$ as suspension. By \cite[Theorem 4.4.1]{Buchweitz}, the map $\stMCM (R) \to \sing (R)$ sending a module to its stalk complex is an equivalence of triangulated categories. Now let $\Kac (\proj R)$ be the homotopy category of acyclic unbounded complexes of finitely generated free $R$-modules. With the usual shifting of complexes as suspension, this category admits a triangulated structure (for any ring) analogous to the one for the homotopy category of matrix factorizations; the distinguished triangles are those isomorphic to standard triangles defined using mapping cones. Again by \cite[Theorem 4.4.1]{Buchweitz}, the map $\Kac(\proj R) \to \stMCM (R)$ sending a complex to the image of its zeroth differential is an equivalence of triangulated categories.

The notion of support varieties for $R$-modules extends to the singularity category and the homotopy category of acyclic complexes of free modules. Given an object $M$ in $\sing (R)$, there is, by the above, a maximal Cohen-Macaulay $R$-module $X_M$ whose stalk complex is isomorphic in $\sing (R)$ to $M$. On the other hand, if $M$ is an object in $\Kac(\proj R)$, then the image $X_M$ of its zeroth differential is also a maximal Cohen-Macaulay module. In either case, we define the support variety $\V_R(M)$ of $M$ to be the variety of the module $X_M$, i.e.\
$$\V_R(M) \stackrel{\text{def}}{=} \V_R(X_M).$$

We end this section with two lemmas we need in order to obtain triangulated subcategories of $\sing (R)$ and $\Kac(\proj R)$ defined in terms of these support varieties. First, consider a short exact sequence
$$\xymatrix@C=30pt{
0 \ar[r] & M_1 \ar[r] & M_2 \ar[r] & M_3 \ar[r] & 0 }$$
of $R$-modules. Given any $\overline{x}$ in $V = (\bm{t}) / \n (\bm{t} )$, the sequence is also exact when viewed as a sequence of $Q/(x)$-modules. If one of the three modules has infinite projective dimension over $Q/(x)$, then the same must be true for at least one of the other two modules. Consequently, support varieties are ``subadditive'' on short exact sequences, in the sense that
$$\V_R(M_u) \subseteq V_R(M_v) \cup V_R(M_w)$$
whenever $\{ u,v,w \} = \{ 1,2,3 \}$. Moreover, if one of the modules has finite projective dimension over $R$, then the varieties of the other two modules must be equal. In particular, the variety of a module $M$ coincides with that of its syzygy $\Omega_R^1(M)$. This follows from the fact that for every $\overline{x} \in V$, the $Q/(x)$-module $R$ has finite projective dimension, being a factor of  $Q/(x)$ by a regular sequence. The first lemma shows that the analogues of these two properties hold for distinguished triangles in the singularity category and the homotopy category of acyclic complexes.

\begin{lemma}\label{Lem:subadditive}
For every distinguished triangle
$$\xymatrix@C=30pt{
M_1 \ar[r] & M_2 \ar[r] & M_3 \ar[r] & \Sigma M_1}$$
in both $\sing (R)$ and $\Kac(\proj R)$, the inclusion $\V_R(M_u) \subseteq V_R(M_v) \cup V_R(M_w)$ holds whenever $\{ u,v,w \} = \{ 1,2,3 \}$. In particular, the equality $\V_R(M) = \V_R( \Sigma M)$ holds for all objects $M$.
\end{lemma}

\begin{proof}
\sloppy It suffices to prove this for the homotopy category $\Kac(\proj R)$. We only need to consider a standard triangle
$$\xymatrix@C=30pt{
M \ar[r]^{[\theta]} & N \ar[r]^{[i_{\theta}]} & C_{\theta} \ar[r]^{[\pi_{\theta}]} & \Sigma M }$$
where $C_{\theta}$ is the mapping cone of $\theta$; here we use the same notation as for the standard triangles in the homotopy category of matrix factorizations. From the triangle we obtain a short exact sequence
$$\xymatrix@C=30pt{
0 \ar[r] & N \ar[r]^{i_{\theta}} & C_{\theta} \ar[r]^{\pi_{\theta}} & \Sigma M \ar[r] & 0 }$$
of complexes of $R$-modules, and in turn a short exact sequence
$$\xymatrix@C=30pt{
0 \ar[r] & \Im d^N_0 \ar[r]^{i_{\theta}} & \Im d^{C_{\theta}}_0 \ar[r]^{\pi_{\theta}} & \Im d^{\Sigma M}_0 \ar[r] & 0 }$$
of images of the zeroth differentials. The $R$-module $\Im d^{\Sigma M}_0$ is the end term of a short exact sequence in which the other end term is $\Im d^M_0$, and where the middle term is a free $R$-module. Thus $\V_R( \Im d^{\Sigma M}_0 ) = \V_R( \Im d^{M}_0 )$, and so the ``subadditivity'' of support varieties for the short exact sequence of images of the zeroth differentials carries over to the standard triangle we started with. This shows that a distinguished triangle 
$$\xymatrix@C=30pt{
M_1 \ar[r] & M_2 \ar[r] & M_3 \ar[r] & \Sigma M_1}$$
in $\Kac (\proj R)$ gives rise to an inclusion $\V_R(M_u) \subseteq V_R(M_v) \cup V_R(M_w)$ whenever $\{ u,v,w \} = \{ 1,2,3 \}$.

Note that we proved along the way that support varieties are stable under the suspension functor, since  $\V_R( \Im d^{\Sigma M}_0 ) = \V_R( \Im d^{M}_0 )$. This also follows from the inclusions we obtain for distinguished triangles. Namely, for every object $M \in \Kac (\proj R)$, there is a distinguished triangle
$$\xymatrix@C=30pt{
M \ar[r] & 0 \ar[r] & \Sigma M \ar[r]^{-1_{\Sigma M}} & \Sigma M }$$
and so since the support variety of the zero object is trivial we see that $\V_R(M) = \V_R( \Sigma M)$.
\end{proof}

Now we define, in terms of support varieties, the triangulated subcategories of $\sing(R)$ and $\Kac(\proj R)$ we shall be concerned with. Namely, given a subset $W \subset V$, consider the subcategories
$$\{ M \in \sing(R) \mid V_R(M) \subseteq W \}$$
$$\{ M \in \Kac(\proj R) \mid V_R(M) \subseteq W \}.$$
The following lemma shows that these are indeed triangulated subcategories of $\sing(R)$ and $\Kac(\proj R)$, respectively; in fact, they are \emph{thick} subcategories, that is, triangulated subcategories closed under direct summands.

\begin{lemma}\label{Lem:thick}
For every subset $W \subset V$, the subcategories
$$\{ M \in \sing(R) \mid V_R(M) \subseteq W \}$$
$$\{ M \in \Kac(\proj R) \mid V_R(M) \subseteq W \}$$
of $\sing(R)$ and $\Kac(\proj R)$ are thick.
\end{lemma}

\begin{proof}
Consider either of the two subcategories. It is obviously closed under isomorphisms, since isomorphic objects have the same support variety. Moreover, since $\V_R(M) = \V_R( \Sigma M )$ by Lemma \ref{Lem:subadditive}, we see that $M$ belongs to the subcategory if and only if its suspension $\Sigma M$ does. Now suppose that
$$\xymatrix@C=30pt{
M_1 \ar[r] & M_2 \ar[r] & M_3 \ar[r] & \Sigma M_1}$$
is a distinguished triangle, and that two of the objects, say $M_u$ and $M_v$, belong to the subcategory. Using Lemma \ref{Lem:subadditive} again, we obtain the inclusion 
$$\V_R(M_w) \subseteq V_R(M_u) \cup V_R(M_v) \subseteq W,$$
hence the object $M_w$ also belongs to the subcategory. This shows that the subcategory is triangulated.

Finally, let $M$ be an object in the subcategory, and $N$ a direct summand of $M$. Then there exists an object $N'$ such that $M$ is isomorphic to $N \oplus N'$, and so $\V_R(M) = \V_R(N) \cup \V_R(N')$. This gives $\V_R(N) \subseteq \V_R(M) \subseteq W$, hence $N$ belongs to the subcategory. The subcategory is therefore thick.
\end{proof}

\section{equivalences}\label{Sec:main}

The main result in this paper shows that we can realize certain triangulated subcategories of the singularity category of a complete intersection as homotopy categories of matrix factorizations. The triangulated subcategories in question arise from support varieties, as in Lemma \ref{Lem:thick}. However, first we show that for \emph{any} commutative ring and non-zerodivisor, the homotopy category of matrix factorizations embeds into the homotopy category of totally acyclic complexes of finitely generated projective modules over the factor ring. This result is analogous to \cite[Theorem 1]{Orlov2} and \cite[Example B.5]{BurkeWalker2}, and will be established through a series of results. 

Recall first that over a commutative ring $R$, a complex $M$ of finitely generated projective $R$-modules is \emph{totally acyclic} if both $M$ and the dualized complex $\Hom_R(M,R)$ are acyclic. These complexes form a thick subcategory $\Ktac (\proj R)$ of $\Kac (\proj R)$. Namely, a standard triangle in $\Kac (\proj R)$ gives rise to a (split) short exact sequence of dualized complexes, and in turn a long exact sequence in homology. Consequently, if two of the dualized complexes are acyclic, then so is the third one. If the ring $R$ is a local Gorenstein ring (for example a complete intersection), then acyclic complexes are automatically totally acyclic; since $R$ has finite injective dimension, a dualized acyclic complex of free modules cannot have nonzero homology. Thus $\Ktac (\proj R) = \Kac (\proj R)$ for local Gorenstein rings, but in general an acyclic complex need not be totally acyclic; cf.\ \cite{JorgensenSega}.

\begin{lemma}\label{Lem:functor}
Let $S$ be a commutative ring, $x \in S$ a non-zerodivisor, and denote by $R$ the factor ring $S/(x)$. To a matrix factorization $(F,G, \phi, \psi)$ in $\MF(S,x)$, assign the sequence
$$\xymatrix@C=30pt{
\cdots \ar[r] & F/xF \ar[r]^{\overline{\phi}} & G/xG \ar[r]^{\overline{\psi}} & F/xF \ar[r]^{\overline{\phi}} & G/xG \ar[r] & \cdots}$$
of finitely generated free $R$-modules. This assignment induces a triangle functor
$$T \colon \HMF(S,x) \to \Ktac(\proj R).$$
\end{lemma}

\begin{proof}
Reducing the matrix factorization $(F,G,\phi,\psi)$ modulo $x$ gives a sequence
$$\xymatrix@C=30pt{
F/xF \ar[r]^{\overline{\phi}} & G/xG \ar[r]^{\overline{\psi}} & F/xF }$$
which is easily seen to be exact; this uses the equalities $\psi \circ \phi = x \cdot 1_F$ and $\phi \circ \psi = x \cdot 1_G$, together with the fact that $x$ is a non-zerodivisor. Thus we obtain an acyclic complex
$$\xymatrix@C=30pt{
M: \cdots \ar[r] & F/xF \ar[r]^{\overline{\phi}} & G/xG \ar[r]^{\overline{\psi}} & F/xF \ar[r]^{\overline{\phi}} & G/xG \ar[r] & \cdots}$$
of finitely generated free $R$-modules. We must show that the complex $\Hom_R(M,R)$ is also acyclic.  

Fix bases for the free $S$-modules $F$ and $G$, and view the maps $\phi$ and $\psi$ as matrices with elements in $S$. Now dualize the original matrix factorization  $(F,G,\phi,\psi)$ and obtain a new matrix factorization
$$\xymatrix@C=30pt{
\Hom_S(F,S) & \Hom_S(G,S) \ar[l]_{\phi^*} & \Hom_S(F,S) \ar[l]_{\psi^*}}$$
in $\MF(S,x)$. Using the canonical isomorphism $\Hom_S(L,S) \simeq L$ for a free $S$-module $L$, one checks easily that this dualized matrix factorization is isomorphic to the matrix factorization
$$\xymatrix@C=30pt{
F & G \ar[l]_{\phi^T} & F \ar[l]_{\psi^T}}$$
in $\MF(S,x)$, where the transposes of the original matrices are the new maps. By the above, reducing modulo $x$ gives an acyclic complex
$$\xymatrix@C=30pt{
N: \cdots & F/xF \ar[l] & G/xG \ar[l]_{\overline{\phi^T}} & F/xF \ar[l]_{\overline{\psi^T}} & G/xG \ar[l]_{\overline{\phi^T}} & \cdots \ar[l] }$$
of free $R$-modules. 

Now consider the complex $M$. The maps are matrices with entries in $R$, and so the above argument shows that the complex 
$$\xymatrix@C=30pt{
\cdots & F/xF \ar[l] & G/xG \ar[l]_{(\overline{\phi})^T} & F/xF \ar[l]_{(\overline{\psi})^T} & G/xG \ar[l]_{(\overline{\phi})^T} & \cdots \ar[l] }$$
is isomorphic to the dualized complex $\Hom_R(M,R)$. But $( \overline{\rho} )^T = \overline{ \rho^T}$ for any matrix $\rho$ over $S$, hence $\Hom_R(M,R)$ is isomorphic to $N$. Since $N$ is acyclic, so is $\Hom_R(M,R)$, and this shows that $M$ is totally acyclic.

 Reducing a morphism of matrix factorizations in $\MF(S,x)$ modulo $x$ gives a morphism of totally acyclic complexes of free $R$-modules. Moreover, by reducing a homotopy between two morphisms in $\MF(S,x)$, we obtain a homotopy between the two morphisms of complexes. Thus $T$ is a functor from $\HMF(S,x)$ to $\Ktac( \proj R)$. From the similarity between the constructions of the standard triangles in these two categories, it is evident that $T$ is a triangle functor.
\end{proof}

First we show that $T$ is faithful. 

\begin{proposition}\label{Prop:faithful}
The triangle functor $T$ in \emph{Lemma \ref{Lem:functor}} is faithful.
\end{proposition}

\begin{proof}
Let $\theta$ be a morphism
$$\xymatrix@C=30pt@R=20pt{
F_1 \ar[r]^{\phi_1} \ar[d]^{f} & G_1 \ar[r]^{\psi_1}  \ar[d]^{g} & F_1  \ar[d]^{f} \\
F_2 \ar[r]^{\phi_2} & G_2 \ar[r]^{\psi_2} & F_2}$$
in $\MF(S,x)$, and suppose that $T$ maps its equivalence class $[\theta]$ in $\HMF(S,x)$ to zero in $\Ktac( \proj R)$. Thus when reducing $\theta$ modulo $x$, the resulting morphism of totally acyclic complexes is nullhomotopic over $R$. Consider a section
$$\xymatrix@C=30pt{
F_1/xF_1 \ar[r]^{\overline{\phi}_1} \ar[d]^{\overline{f}} & G_1/xG_1 \ar[r]^{\overline{\psi}_1} \ar[d]^{\overline{g}} \ar[dl]_{\overline{s}_1} & F_1/xF_1 \ar[r]^{\overline{\phi}_1} \ar[d]^{\overline{f}} \ar[dl]_{\overline{t}} & G_1/xG_1 \ar[d]^{\overline{g}} \ar[dl]_{\overline{s}_2} \\
F_2/xF_2 \ar[r]^{\overline{\phi}_2} & G_2/xG_2 \ar[r]^{\overline{\psi}_2} & F_2/xF_2 \ar[r]^{\overline{\phi}_2} & G_2/xG_2 }$$
of such a nullhomotopy (the homotopy is not necessarily periodic, that is, $\overline{s}_1$ may not equal $\overline{s}_2$), and choose liftings of the three diagonal maps to $S$-homomorphisms
$$s_1 \colon G_1 \to F_2, \hspace{5mm} t \colon F_1 \to G_2, \hspace{5mm} s_2 \colon G_1 \to F_2.$$
For every $a \in F_1$, the homotopy implies the existence of an element $b_a \in F_2$ such that
$$f(a) - s_2 \circ \phi_1 (a) - \psi_2 \circ t(a) = x \cdot b_a,$$
and this element is unique since $x$ is a non-zerodivisor. Similarly, for every $u \in G_1$ there exists a unique element $v_u \in G_2$ with
$$g (u) - t \circ \psi_1 (u) - \phi_2 \circ s_1 (u) = x \cdot v_u.$$
The maps
\begin{eqnarray*}
p \colon F_1 \to F_2, \hspace{5mm} a \mapsto b_a \\
q \colon G_1 \to G_2, \hspace{5mm} u \mapsto v_u
\end{eqnarray*}
are therefore well defined $S$-homomorphisms. Now modify $t$ to a new map $t' \colon F_1 \to G_2$ defined by 
$$t' = t + \phi_2 \circ p.$$
We shall show that $(s_2,t')$ is a nullhomotopy on $\theta$:
$$\xymatrix@C=40pt{
F_1 \ar[r]^{\phi_1} \ar[d]^{f} & G_1 \ar[r]^{\psi_1} \ar[dl]_{s_2} \ar[d]^{g} & F_1  \ar[d]^{f} \ar[dl]_{t'} \\
F_2 \ar[r]^{\phi_2} & G_2 \ar[r]^{\psi_2} & F_2}$$

The definition of $t'$ gives
\begin{eqnarray*}
f - s_2 \circ \phi_1 - \psi_2 \circ t' & = & f - s_2 \circ \phi_1 - \psi_2 \circ \left ( t + \phi_2 \circ p \right ) \\
& = &  f - s_2 \circ \phi_1 - \psi_2 \circ t - x \cdot p \\
& = & 0.
\end{eqnarray*}
Next, consider the equality $f - s_2 \circ \phi_1 - \psi_2 \circ t = x \cdot p$ from above. Composing with $\psi_1$ gives
\begin{eqnarray*}
x \cdot p \circ \psi_1 & = & f \circ \psi_1 - s_2 \circ \phi_1 \circ \psi_1 - \psi_2 \circ t \circ \psi_1 \\
& = & \psi_2 \circ g - x \cdot s_2 - \psi_2 \circ t \circ \psi_1 \\
& = & \psi_2 \circ \left ( g- t \circ \psi_1 \right ) - x \cdot s_2 \\
& = & \psi_2 \circ \left ( \phi_2 \circ s_1 + x \cdot q \right ) - x \cdot s_2 \\
& = & x \cdot \left ( s_1-s_2 + \psi_2 \circ q \right ),
\end{eqnarray*}
where we have also used the equality $g - t \circ \psi_1 - \phi_2 \circ s_1 = x \cdot q$. Since $x$ is a non-zerodivisor, we see that
$$p \circ \psi_1 = s_1-s_2+ \psi_2 \circ q.$$
This, in turn, gives
\begin{eqnarray*}
g - t' \circ \psi_1 - \phi_2 \circ s_2 & = & g - \left ( t + \phi_2 \circ p \right ) \circ \psi_1 - \phi_2 \circ \left ( s_1 + \psi_2 \circ q - p \circ \psi_1 \right ) \\
& = & g - t \circ \psi_1 - \phi_2 \circ s_1 - x \cdot q \\ 
& = & 0.
\end{eqnarray*}
We have proved that $(s_2,t')$ is a nullhomotopy on $\theta$, hence $[ \theta ]$=0 in $\HMF(S,x)$. This shows that the functor $T$ is faithful.
\end{proof}

Next, we show that $T$ is full. In the proof, note the similarities with the arguments in the proof of Proposition \ref{Prop:faithful}.

\begin{proposition}\label{Prop:full}
The triangle functor $T$ in \emph{Lemma \ref{Lem:functor}} is full.
\end{proposition}

\begin{proof}
Let $(F_1,G_1,\phi_1,\psi_1)$ and $(F_2,G_2,\phi_2,\psi_2)$ be matrix factorizations in $\MF(S,x)$, and suppose that 
$$\xymatrix@C=30pt{
\cdots \ar[r] & F_1/xF_1 \ar[r]^{\overline{\phi}_1} \ar[d]^{\overline{f}_1} & G_1/xG_1 \ar[r]^{\overline{\psi}_1} \ar[d]^{\overline{g}_1} & F_1/xF_1 \ar[r]^{\overline{\phi}_1} \ar[d]^{\overline{f}_0} & G_1/xG_1 \ar[r] \ar[d]^{\overline{g}_{0}} & \cdots \\
\cdots \ar[r] & F_2/xF_2 \ar[r]^{\overline{\phi}_2} & G_2/xG_2 \ar[r]^{\overline{\psi}_2} & F_2/xF_2 \ar[r]^{\overline{\phi}_2} & G_2/xG_2 \ar[r] & \cdots}$$
is a chain map $\eta$ of totally acyclic complexes over $R$, representing a morphism $[ \eta ]$ in $\Ktac ( \proj R)$. When lifting a section to $S$, we obtain a diagram
$$\xymatrix@C=30pt@R=20pt{
G_1 \ar[r]^{\psi_1}\ar[d]^{g_1} & F_1 \ar[r]^{\phi_1}\ar[d]^{f_0} & G_1 \ar[d]^{g_{0}} \\
G_2\ar[r]^{\psi_2} & F_2 \ar[r]^{\phi_2} & G_2 }$$
where the vertical maps are chosen liftings. Now let $u$ be an element in $G_1$, and $a$ an element in $F_1$. Since the diagram commutes when we reduce modulo $x$, there exist elements $v_u \in F_2$ and $b_a \in G_2$ such that
\begin{eqnarray*}
\phi_2 \circ f_0(a) - g_{0} \circ \phi_1(a) & = & x \cdot b_a \\
\psi_2 \circ g_1(u) - f_0 \circ \psi_1(u) & = & x \cdot v_u,
\end{eqnarray*}
and these elements are unique since $x$ is a non-zerodivisor. It follows that the maps
\begin{eqnarray*}
\alpha \colon F_1 \to G_2, \hspace{5mm} a \mapsto b_a \\
\beta \colon G_1 \to F_2, \hspace{5mm} u \mapsto v_u
\end{eqnarray*}
are well defined $S$-homomorphisms, giving equalities 
\begin{eqnarray*}
\phi_2 \circ f_0 - g_{0} \circ \phi_1 & = & x \cdot \alpha \\
\psi_2 \circ g_1 - f_0 \circ \psi_1 & = & x \cdot \beta.
\end{eqnarray*}
The first equality gives
$$x \cdot \psi_2 \circ \alpha \circ \psi_1 = \psi_2 \circ \left ( \phi_2 \circ f_0 - g_{0} \circ \phi_1 \right ) \circ \psi_1 = x \cdot f_0 \circ \psi_1 - x \cdot \psi_2 \circ g_0,$$
and so
$$\psi_2 \circ \alpha \circ \psi_1 = f_0 \circ \psi_1 - \psi_2 \circ g_0$$
since $x$ is a non-zerodivisor.

Now consider the diagram
$$\xymatrix@C=30pt@R=20pt{
G_1 \ar[r]^{\psi_1}\ar[d]^{g} & F_1 \ar[r]^{\phi_1}\ar[d]^{f} & G_1 \ar[d]^{g} \\
G_2\ar[r]^{\psi_2} & F_2 \ar[r]^{\phi_2} & G_2 }$$
where the vertical maps are defined by
\begin{eqnarray*}
f & = & f_0 - \psi_2 \circ \alpha + \beta \circ \phi_1 \\
g & = & g_0 + \phi_2 \circ \beta.
\end{eqnarray*}
Using the equalities established above, we obtain 
\begin{eqnarray*}
\psi_2 \circ g & = & \psi_2 \circ g_0 + \psi_2 \circ \phi_2 \circ \beta \\
& = & \left ( f_0 \circ \psi_1 - \psi_2 \circ \alpha \circ \psi_1 \right ) + x \cdot \beta \\
& = & f_0 \circ \psi_1 - \psi_2 \circ \alpha \circ \psi_1 + \beta \circ \phi_1 \circ \psi_1 \\
& = & f \circ \psi_1 \\
\phi_2 \circ f & = & \phi_2 \circ f_0 - \phi_2 \circ \psi_2 \circ \alpha + \phi_2 \circ \beta \circ \phi_1 \\
& = & \left ( \phi_2 \circ f_0 - x \cdot \alpha \right ) + \phi_2 \circ \beta \circ \phi_1 \\
& = & g_0 \circ \phi_1 + \phi_2 \circ \beta \circ \phi_1 \\
& = & g \circ \phi_1,
\end{eqnarray*}
showing that the diagram commutes. Thus the pair $\theta = (f,g)$ is a morphism $\theta \colon (F_1,G_1,\phi_1,\psi_1) \to (F_2,G_2,\phi_2,\psi_2)$ of matrix factorizations in $\MF(S,x)$; we shall show that $T( [ \theta ]) = [ \eta ]$.

The morphism $T([ \theta ])$ in $\Ktac( \proj R)$ is represented by the two-periodic chain map
$$\xymatrix@C=30pt{
\cdots \ar[r] & F_1/xF_1 \ar[r]^{\overline{\phi}_1} \ar[d]^{\overline{f}} & G_1/xG_1 \ar[r]^{\overline{\psi}_1} \ar[d]^{\overline{g}} & F_1/xF_1 \ar[r]^{\overline{\phi}_1} \ar[d]^{\overline{f}} & G_1/xG_1 \ar[r] \ar[d]^{\overline{g}} & \cdots \\
\cdots \ar[r] & F_2/xF_2 \ar[r]^{\overline{\phi}_2} & G_2/xG_2 \ar[r]^{\overline{\psi}_2} & F_2/xF_2 \ar[r]^{\overline{\phi}_2} & G_2/xG_2 \ar[r] & \cdots}$$
of totally acyclic complexes.  We must show that this chain map is homotopic to the chain map $\eta$ we started with. Consider therefore the diagram
$$\xymatrix@C=30pt{
\cdots \ar[r] & F_1/xF_1 \ar[r]^{\overline{\phi}_1} \ar[d]^{\overline{f} - \overline{f}_1} & G_1/xG_1 \ar[r]^{\overline{\psi}_1} \ar[d]^{\overline{g} - \overline{g}_1} & F_1/xF_1 \ar[r]^{\overline{\phi}_1} \ar[d]^{\overline{f} - \overline{f}_0} \ar[dl]_{- \overline{\alpha}} & G_1/xG_1 \ar[r] \ar[d]^{\overline{g} - \overline{g}_0} \ar[dl]_{\overline{\beta}} & \cdots \\
\cdots \ar[r] & F_2/xF_2 \ar[r]^{\overline{\phi}_2} & G_2/xG_2 \ar[r]^{\overline{\psi}_2} & F_2/xF_2 \ar[r]^{\overline{\phi}_2} & G_2/xG_2 \ar[r] & \cdots}$$
of $R$-modules and maps. The very definition of the $S$-homomorphism $f$ gives
$$\overline{f} - \overline{f}_0 = - \overline{\psi}_2 \circ \overline{\alpha} + \overline{\beta} \circ \overline{\phi}_1,$$
hence the diagram displays the ``zeroth part'' of a possible nullhomotopy. Since the lower complex is acyclic and the upper complex consists of free $R$-modules, a standard argument allows us to complete the nullhomotopy to the left. Since the complexes are totally acyclic, they remain exact when we apply the functor $\Hom_R(-,R)$ to the diagram. As above, we may complete the nullhomotopy in this new diagram to the right, and when we apply $\Hom_R(-,R)$ once more, we obtain a complete nullhomotopy in the original diagram. This shows that $T( [ \theta ]) = [ \eta ]$ in $\Ktac ( \proj R)$, and so we have proved that the functor $T$ is full.
\end{proof}

Combining the previous results, we obtain the following.

\begin{theorem}\label{Thm:embedding}
Let $S$ be a commutative ring, $x \in S$ a non-zerodivisor, and denote by $R$ the factor ring $S/(x)$. To a matrix factorization $(F,G, \phi, \psi)$ in $\MF(S,x)$, assign the complex
$$\xymatrix@C=30pt{
\cdots \ar[r] & F/xF \ar[r]^{\overline{\phi}} & G/xG \ar[r]^{\overline{\psi}} & F/xF \ar[r]^{\overline{\phi}} & G/xG \ar[r] & \cdots}$$
of free $R$-modules, and assign to a morphism in $\MF(S,x)$ the obvious morphism of complexes. This assignment induces a triangle functor
$$T \colon \HMF(S,x) \to \Ktac( \proj R)$$
which is fully faithful.
\end{theorem}

Having established this general result, we now prove our main result for complete intersections. Thus let $(Q, \n, k)$ be a regular local ring and $\bm{t} = t_1, \dots, t_c$ a regular sequence contained in $\n^2$. Denote by $R$ the codimension $c$ complete intersection $Q/ ( \bm{t} )$, and by $V$ the $c$-dimensional $k$-vector space $(\bm{t}) / \n (\bm{t} )$. The result shows that the thick subcategories of $\sing(R)$ defined in terms of complements of $(c-1)$-dimensional subspaces of $V$ (cf.\ Lemma \ref{Lem:thick}) are equivalent to homotopy categories of matrix factorizations over complete intersections of codimension $c-1$. 

Recall first that if $\overline{x}$ is any nonzero element of $V$, then when completing to a basis $\overline{x}, \overline{x}_1, \dots, \overline{x}_{c-1}$ and lifting to $Q$, we obtain a regular sequence $x, x_1, \dots, x_{c-1}$. The image of $x$ is a non-zerodivisor in the codimension $c-1$ complete intersection $S = Q/(x_1, \dots, x_{c-1})$, and $R = S/(x)$.

\begin{theorem}\label{Thm:main}
Let $(Q, \n, k)$ be a regular local ring and $\bm{t} = t_1, \dots, t_c$ a regular sequence contained in $\n^2$. Denote by $R$ the complete intersection $Q/ ( \bm{t} )$, by $V$ the $c$-dimensional $k$-vector space $(\bm{t}) / \n (\bm{t} )$, and let $\overline{x}$ be a nonzero element in $V$. Complete this element to a basis $\overline{x}, \overline{x}_1, \dots, \overline{x}_{c-1}$, denote by $W$ the subspace of $V$ spanned by $\overline{x}_1, \dots, \overline{x}_{c-1}$, and by $S$ the codimension $c-1$ complete intersection $Q/(x_1, \dots, x_{c-1})$. Then the map
\begin{eqnarray*}
\HMF(S,x) & \to & \{ M \in \sing(R) \mid \V_R(M) \subseteq \{ 0 \} \cup V \setminus W \} \\
(F,G,\phi,\psi)  & \mapsto & \Coker \psi
\end{eqnarray*}
is an equivalence of triangulated categories, where we view $\Coker \psi$ as a stalk complex.
\end{theorem}

\begin{proof}
If $c=1$, then $S=Q$, $W=0$, and $R= Q/(x)$. In this case, the statement says that taking cokernels induces a triangle equivalence between $\HMF(Q,x)$ and $\sing(R)$. This result was noted by Buchweitz in \cite{Buchweitz}, and proved explicitly in \cite[Theorem 3.9]{Orlov1}. We may therefore assume that $c$ is at least $2$. Moreover, since $R$ is Gorenstein, we may replace the category $\sing (R)$ in the statement by $\Kac ( \proj R ) = \Ktac ( \proj R )$.

Consider the fully faithful triangle functor
$$T \colon \HMF(S,x) \to \Kac( \proj R)$$
from Theorem \ref{Thm:embedding}, and let $(F,G,\phi,\psi)$ be a matrix factorization in $\MF(S,x)$. The image of this matrix factorization under $T$ is the acyclic complex
$$\xymatrix@C=25pt{
M \colon & \cdots \ar[r] & F/xF \ar[r]^{\overline{\phi}} & G/xG \ar[r]^{\overline{\psi}} & F/xF \ar[r]^{\overline{\phi}} & G/xG \ar[r] & \cdots}$$
in $\Kac( \proj R)$. Since $\psi \circ \phi = x \cdot 1_F$, we see that $x$ annihilates $\Coker \psi$, making it an $R$-module; we denote this module by $X_M$. This module is the zeroth differential in $M$, since
\begin{eqnarray*}
\Coker \overline{\psi} & = & \left ( F/xF \right ) / \Im \overline{\psi} \\
& = &  \left ( F/xF \right ) / \left ( [ \Im \psi + xF ] / xF \right ) \\
& = & \left ( F/xF \right ) / \left ( [ \Im \psi + \psi \circ \phi(F) ] / xF \right ) \\
& = & \left ( F/xF \right ) / \left ( \Im \psi / xF \right ) \\
& \simeq & F / \Im \psi.
\end{eqnarray*}
By definition, the support variety of the complex $M$ equals that of the $R$-module $X_M$, i.e.\ $\V_R(M) = \V_R(X_M)$.

The compositions of the two maps in a matrix factorisation of a non-zerodivisor are injective, hence so are the maps themselves. The sequence
$$0 \to G \xrightarrow{\psi} F \to X_M \to 0$$
of $S$-modules is therefore exact, and so $\pd_S X_M$ is finite. Now suppose that $\overline{y}$ is a nonzero element in the subspace $W$ of $V$ (note that $W$ is nonzero since $c \ge 2$), and complete to a basis $\overline{y}, \overline{y}_1, \dots, \overline{y}_{c-2}$ of $W$ (if $c \ge 3$). In the terminology of \cite{BerghJorgensen}, the complete intersections $S= Q/(x_1, \dots, x_{c-1})$ and $Q/(y, y_1, \dots, y_{c-2})$ are equivalent, and it follows therefore from \cite[Proposition 3.4]{BerghJorgensen} that the projective dimension of $X_M$ as a module over the latter is finite. Then $\pd_{Q/(y)}X_M$ is also finite, since, if $c \ge 3$, the ring $Q/(y, y_1, \dots, y_{c-2})$ is a quotient of $Q/(y)$ by a regular sequence. This shows that $\overline{y}$ is not an element in $\V_R(X_M)$, and consequently $\V_R(X_M) \subseteq \{ 0 \} \cup V \setminus W$.

Conversely, let $M$ be an indecomposable object in $\Kac ( \proj R)$ with $\V_R(M) \subseteq  \{ 0 \} \cup V \setminus W$. We may assume that $M$ is a minimal acyclic complex; if not, then it is isomorphic in $\Kac ( \proj R)$ to such a complex. Consider the image $X_M$ of the zeroth differential in $M$. This is a maximal Cohen-Macaulay $R$-module, and by definition $\V_R(M)$ equals $\V_R(X_M)$. Since $W \cap \V_R(X_M) = \{ 0 \}$, it follows from \cite[Theorem 3.1]{BerghJorgensen} that $\V_S(X_M)$ is trivial, and then in turn from \cite[Theorem 2.5 and Theorem 5.6]{AvramovBuchweitz} that $X_M$ has finite projective dimension over $S$.

\sloppy By \cite[Theorem 4.4]{Avramov}, there exists a matrix factorization $(F,G,\phi,\psi)$ in $\MF(S,x)$ with the property that 
$$\xymatrix@C=25pt{
\cdots \ar[r] & F/xF \ar[r]^{\overline{\phi}} & G/xG \ar[r]^{\overline{\psi}} & F/xF \ar[r]^{\overline{\phi}} & G/xG \ar[r] & \Omega_R^1(X_M) \ar[r] & 0 }$$
is a minimal free resolution of $\Omega_R^1(X_M)$. Thus this module is the image of the first differential in the acyclic complex $T(F,G,\phi,\psi)$, which is minimal. Consequently, when we truncate the complexes $M$ and $T(F,G,\phi,\psi)$ at degree one, they are isomorphic, because the truncation of $M$ is also a minimal free resolution of $\Omega_R^1(X_M)$. Being totally acyclic, this implies that $M$ and $T(F,G,\phi,\psi)$ themselves are isomorphic.
\end{proof}

\begin{corollary}\label{Cor:equivalences1}
Let $(Q, \n, k)$ be a regular local ring and $\bm{t} = t_1, \dots, t_c$ a regular sequence contained in $\n^2$. Denote by $V$ the $c$-dimensional $k$-vector space $(\bm{t}) / \n (\bm{t} )$, and let $\overline{x}$ be a nonzero element in $V$. Furthermore, let $\overline{x}_1, \dots, \overline{x}_{c-1}$ and $\overline{y}_1, \dots, \overline{y}_{c-1}$ be two completions of $\overline{x}$ to bases of $V$, generating the same $(c-1)$-dimensional subspace of $V$. Finally, consider the codimension $c-1$ complete intersections $S = Q/(x_1, \dots, x_{c-1})$ and $S'=Q/(y_1, \dots, y_{c-1})$. Then the triangulated categories $\HMF(S,x)$ and $\HMF(S',x)$ are equivalent. 
\end{corollary}

\begin{proof}
Denote by $W$ the $(c-1)$-dimensional subspace of $V$ generated by $\overline{x}_1, \dots, \overline{x}_{c-1}$. By the theorem, both categories are equivalent to the category $\{ M \in \sing(R) \mid \V_R(M) \subseteq \{ 0 \} \cup V \setminus W \}$.
\end{proof}

\begin{example}
Consider the power series ring $Q= \mathbb C \llbracket x,y \rrbracket$. Take the regular sequence 
$\bm{t} =x^2,y^2$, and consider the $\mathbb C$-vector space $V= (\bm{t}) / \n (\bm{t} )$, where 
$\n = (x,y)$. Then the images in $V$ of the elements $y^2$ and $y^2 - x^3$ of $Q$ both complete the element $x^2 +  \n (\bm{t} )$ to a basis for $V$, and these two elements generate the same subspace of $V$. Hence the triangulated categories $\HMF(Q/(y^2),x^2)$ and $\HMF(Q/(y^2-x^3),x^2)$ are equivalent. Note that the hypersurfaces $Q/(y^2)$ and $Q/(y^2-x^3)$ are not isomorphic: one of them is an integral domain, the other is not.
\end{example}

\end{document}